\documentclass[12pt,reqno,a4paper]{amsart}
\usepackage{hyperref}
\usepackage{enumerate}
\usepackage[left=0.75in, right=0.75in, top= 1in, bottom= 1.75in ]{geometry}
\usepackage{amsmath,amssymb,amsthm,amsfonts}
\usepackage{graphicx}
\usepackage{tikz}
\usetikzlibrary{graphs,graphs.standard}
\usetikzlibrary{shapes}
\usetikzlibrary{plotmarks}
\usetikzlibrary{arrows}
\usetikzlibrary{positioning}
\theoremstyle{definition}
\newtheorem{defn}{Definition}[section]
\newtheorem{fact}[defn]{Fact}
\newtheorem{prop}[defn]{Proposition}
\newtheorem{thm}[defn]{Theorem}

\newtheorem{corr}[defn]{Corollary}
\newtheorem{lem}[defn]{Lemma}
\newtheorem{question}[defn]{Question}
\newtheorem{remark}[defn]{Remark}

\newtheorem{claim}[defn]{claim}

\title[Erd\H{o}s--Dushnik--Miller theorem without AC]{On Erd\H{o}s--Dushnik--Miller theorem without AC}
\author{Amitayu Banerjee}
\address{Alfr\'ed R\'enyi Institute of Mathematics, Reáltanoda utca 13-15, Budapest-1053, Hungary}
\email{banerjee.amitayu@gmail.com}

\author{Alexa Gopaulsingh}
\address{Department of Logic, Institute of Philosophy, E\"otv\"os Lor\'and University, Budapest, Hungary}
\email{alexa279e@gmail.com}
\date{}

\makeatletter
\@namedef{subjclassname@2020}{\textup{2020} Mathematics Subject Classification}
\makeatother
\begin{document}
\begin{abstract}
In $\mathsf{ZFA}$ (Zermelo-Fraenkel set theory with the Axiom of Extensionality weakened to allow the existence
of atoms), we prove that the strength of the proposition $\mathsf{EDM}$ (``If $G=(V_{G}, E_{G})$ is a graph such that $V_{G}$ is uncountable, then for all coloring $f:[V_{G}]^{2}\rightarrow \{0,1\}$ either there is an uncountable set monochromatic in color $0$, or there is a countably infinite set monochromatic in color 1") is {\em strictly between}  $\mathsf{DC_{\aleph_{1}}}$ (where $\mathsf{DC_{\aleph_{1}}}$ is Dependent Choices for $\aleph_{1}$, a weak choice form stronger than Dependent Choices ($\mathsf{DC}$)) and Kurepa's principle (``Any partially ordered set such that all of its antichains are finite and all of its chains are countable is countable"). Among other new results, we 
study the relations of $\mathsf{EDM}$ with $\mathsf{BPI}$ (Boolean Prime Ideal Theorem), $\mathsf{RT}$ (Ramsey's  Theorem), De Bruijn–Erd\H{o}s theorem for $n$-colorings, K\H{o}nig's Lemma and several other weak choice forms. 
Moreover, we answer a  part of a question raised by Lajos Soukup.
\end{abstract}
\maketitle
\section{introduction}

In 1941, Dushnik and Miller established the proposition 
{\em ``If $\kappa$ is an uncountable cardinal, then for all coloring $f:[\kappa]^{2}\rightarrow \{0,1\}$, either there is a set of cardinality $\kappa$ monochromatic in color 0, or there is a countably infinite set monochromatic in color 1"} in $\mathsf{ZFC}$, and gave credit to Erd\H{o}s for the proof of the result for the case in which $\kappa$ is a singular cardinal—see \cite[Theorem 5.22 and footnote 6 on page 606]{DM1941}.
The above result is uniformly known as Erd\H{o}s--Dushnik--Miller theorem. In $\mathsf{ZFC}$, the theorem applies to prove the proposition {\em ``Every partially ordered set such that all of its chains are finite and all of its antichains are countable is countable"} (abbreviated here as ``$\mathsf{CAC^{\aleph_{0}}}$")
and Kurepa's result on partially ordered sets stated in the
abstract (abbreviated here as ``$\mathsf{CAC_{1}^{\aleph_{0}}}$").
Kurepa \cite{Kur1958} explicitly proved $\mathsf{CAC_{1}^{\aleph_{0}}}$ in $\mathsf{ZFC}$ in response to Sierpi\'{n}ski's question \cite[pages 190–191]{Sie1958}. 
Banerjee \cite{Ban2} and Banerjee and Gyenis \cite{BG2021} studied some relations of $\mathsf{CAC^{\aleph_{0}}}$ and $\mathsf{CAC_{1}^{\aleph_{0}}}$ with weak forms of the Axiom of Choice ($\mathsf{AC}$). Recently, Tachtsis \cite{Tac2022} investigated the deductive strength of $\mathsf{CAC_{1}^{\aleph_{0}}}$ without $\mathsf{AC}$ in more detail. 
Among various results, Tachtsis \cite{Tac2022}  proved that $\mathsf{CAC_{1}^{\aleph_{0}}}$ holds in  $\mathsf{ZF + DC_{\aleph_{1}}}$\footnote{$\mathsf{ZF}$ denotes Zermelo--Fraenkel set theory without $\mathsf{AC}$. Complete definitions of the choice forms will be given in \textbf{section 2}.} and $\mathsf{CAC_{1}^{\aleph_{0}}}$ holds in the Mostowski linearly ordered model (labeled as Model $\mathcal{N}_{3}$ in \cite{HR1998}) as well as the basic Fraenkel model (labeled as Model $\mathcal{N}_{1}$ in \cite{HR1998}).
Inspired by the research work of \cite{Tac2022}, we study the deductive strength of $\mathsf{EDM}$ without $\mathsf{AC}$. Let $X$ be a set. We note that without $\mathsf{AC}$, there are two definitions of {\em uncountable} sets: 
\begin{enumerate}
    \item {\em $X$ is uncountable} if $\vert X\vert \not\leq \aleph_{0}$ (i.e., there is no injection from $X$ into $\aleph_{0}$).
    \item {\em $X$ is uncountable} if $\aleph_{0} < \vert X\vert$ (i.e., there is an injection from $\aleph_{0}$ into $X$ and there is no injection from $X$ into $\aleph_{0}$).
\end{enumerate}
We note that all the results in this paper are obtained with the first definition of uncountable sets. Lajos Soukup asked the following question. 

\begin{question}
What is the relationship between $\mathsf{CAC_{1}^{\aleph_{0}}}$ and  $\mathsf{CAC^{\aleph_{0}}}$ in $\mathsf{ZF}$ and $\mathsf{ZFA}$?   
\end{question}

\textbf{Main result.} Fix $X\in \{\mathsf{CAC_{1}^{\aleph_{0}}}, \mathsf{CAC^{\aleph_{0}}}\}$. The first author proves that the strength of $\mathsf{EDM}$ is strictly  between $\mathsf{DC_{\aleph_{1}}}$ and $X$, and $\mathsf{CAC_{1}^{\aleph_{0}}}$ does not imply $\mathsf{CAC^{\aleph_{0}}}$ in $\mathsf{ZFA}$ (cf. \textbf{Theorems 4.1, 4.2, 4.4, Corollary 3.6}).

\textbf{Other results.} The first author observes the following in $\mathsf{ZF}$:
\begin{enumerate}
    \item $\mathsf{CAC^{\aleph_{0}}}$ implies $\mathsf{AC^{\aleph_{0}}_{\aleph_{0}}}$ (Every countably infinite family of countably infinite sets has a choice function). Thus $\mathsf{CAC^{\aleph_{0}}}$ is not provable in $\mathsf{ZF}$ (\textbf{Proposition 3.4}).
    \item $\mathsf{WOAM}$ (Every set is either well-orderable or has an amorphous subset) implies
    $\mathsf{RT}$ for any locally countable connected graph (\textbf{Proposition 3.7}).
    \item $\mathsf{EDM}$ is strictly stronger than $\mathsf{RT}$  (\textbf{Theorem 4.1(3,4)}). 
    
    \item $\mathsf{WOAM}$ + $\mathsf{RT}$ implies $\mathsf{EDM}$ (\textbf{Theorem 4.2(1)}). 
     \item  $\mathsf{CAC^{\aleph_{0}}}+\mathsf{CAC_{1}^{\aleph_{0}}}$ + $\mathsf{A}$(Antichain Principle) does not imply $\mathsf{EDM}$ in $\mathsf{ZFA}$ (\textbf{Theorem 4.4}). 
\end{enumerate}   

\subsection{The uniqueness of algebraic closures, and \L{}o\'{s}'s theorem} Pincus \cite[\textbf{Note 41}]{HR1998} proved that the statement ``If a field has an algebraic closure, then it is unique up to isomorphism" \cite[\textbf{Form 233}]{HR1998} does not imply ``there are no amorphous sets" in $\mathsf{ZFA}$.
Recently, Tachtsis \cite{Tac2019a} constructed a model of $\mathsf{ZFA}$+$\neg \mathsf{AC}$ to prove that $\mathsf{AC^{LO}}$ (Every linearly ordered family of non-empty sets has a choice function) does not imply $\mathsf{LT}$ (if $\mathcal{A}=\langle A, \mathcal{R}^{\mathcal{A}}\rangle$ is a non-trivial relational $\mathcal{L}$-structure over some language $\mathcal{L}$, and $\mathcal{U}$ be an ultrafilter on a non-empty set $I$, then the ultrapower $\mathcal{A}^{I}/\mathcal{U}$ and $\mathcal{A}$ are elementarily equivalent). 

\textbf{Other results.} We observe the following:
\begin{enumerate}
    \item $\mathsf{AC^{LO}}$ + $\mathsf{EDM}$ + \textbf{Form 233} does not imply $\mathsf{LT}$ in $\mathsf{ZFA}$  (\textbf{Theorem 5.2}).
\end{enumerate}
\subsection{Remarks} 
Blass \cite{Bla1977} investigated the strength of $\mathsf{RT}$ in the hierarchy of choice forms. In \textbf{section 6}, applying the above-mentioned results and mainly inspired by the results of \cite{Bla1977}, we remark that $\mathsf{EDM}$ is independent of each of $\mathsf{BPI}$, $\mathsf{KW}$ (Kinna-Wagner Selection Principle), $\mathsf{AC_{WO}}$ (Axiom of Choice for non-empty, well-orderable sets), ``There are no amorphous sets", $n$-coloring theorem (De Bruijn–Erd\H{o}s theorem for $n$-colorings), and $\mathsf{A}$ (Antichain Principle) in $\mathsf{ZFA}$. Moreover, $\mathsf{LT}$ and $\mathsf{EDM}$ are mutually independent in $\mathsf{ZF}$.
In \cite{Tac2022}, Tachtsis proved that $\mathsf{CAC_{1}^{\aleph_{0}}}$ and $\mathsf{DT}$ (Dilworth’s
theorem) are mutually independent in $\mathsf{ZFA}$. A natural question which arises is about the relation of $\mathsf{CAC^{\aleph_{0}}}$ and $\mathsf{EDM}$ with $\mathsf{DT}$. We also remark that $\mathsf{DT}$ is independent of $\mathsf{EDM}$ and $\mathsf{CAC^{\aleph_{0}}}$ in $\mathsf{ZFA}$. 

\section{Basics and diagram of results}
\begin{defn}
Suppose $X$ and $Y$ are two sets. We write:
\begin{enumerate}
    \item $\vert X\vert \leq \vert Y\vert$ or $\vert Y\vert \geq \vert X\vert$, if there is an injection $f : X \rightarrow Y$.
    \item $\vert X\vert = \vert Y\vert$, if there is a bijection $f : X \rightarrow Y$.
    \item $\vert X\vert < \vert Y\vert$ or $\vert Y\vert > \vert X\vert$, if $\vert X\vert \leq \vert Y\vert$ and $\vert X\vert \not= \vert Y\vert$.
    \item If $f:X\rightarrow Y$ is a function, then we denote ``the range of $f$" by $ran(f)$ and ``the domain of $f$" by $dom(f)$.
\end{enumerate}
\end{defn}
\begin{defn}
Let $(P, \leq)$ be a partially ordered set, or ‘poset’ in short. 
A subset $D \subseteq P$ is a {\em chain} if $(D, \leq\restriction D)$ is linearly ordered. A subset $A\subseteq P$ is an {\em antichain}
if no two elements of $A$ are comparable under $\leq$. The size of the largest antichain of $(P, \leq)$ is known as its {\em width}.  A subset $C \subseteq P$ is {\em cofinal} in $P$ if for every $x \in P$ there is an element $c \in C$ such that $x \leq c$. 
A {\em tree} is a connected undirected graph without circuits one of whose vertices is designated as the origin. 
We note that a tree may also be defined as a poset $(P, <)$ with a least element and with the property that for any element $x \in P$ the set of 
predecessors of $x$ is a finite set that is linearly ordered by $<$. 
The number of vertices on the unique path connecting a vertex $v$ with the origin 
is the {\em level} of $v$, denoted by $l(v)$. A vertex $v'$ is a {\em successor} of a vertex $v$ if $v$ and $v'$ are connected by an edge and $l(v') = l(v) + 1$. 
A tree is {\em locally finite} if each vertex has only finitely many successors. An {\em $\omega$-tree} is a locally finite tree with at least one vertex in level $n$ for each $n\in\omega$ (cf. \cite[Note 21]{HR1998}). An inﬁnite set $X$ is {\em amorphous} if $X$ cannot be written as a disjoint union of two infinite subsets. A set $X$ is {\em Dedekind-finite} if $\aleph_{0} \not
\leq \vert X\vert$. Otherwise, $X$ is {\em Dedekind-infinite}. We say that a graph $G=(V_{G}, E_{G})$ is {\em locally countable} if for every $v \in V_{G}$, the set of neighbors of $v$ is countable. The graph $G$ is {\em connected} if any two vertices are joined by a path of finite length.
\end{defn}

\begin{defn}{\textbf{(A list of choice forms).}}

\begin{enumerate}
    \item The {\em Axiom of Choice}, $\mathsf{AC}$ \cite[\textbf{Form 1}]{HR1998}: Every family of non-empty sets has a choice function.
    \item The {\em Boolean Prime Ideal Theorem}, $\mathsf{BPI}$ (\cite[\textbf{Form 14}]{HR1998}): Every Boolean algebra has a prime ideal.
    \item The {\em Kinna-Wagner Selection Principle}, $\mathsf{KW}$ (\cite[\textbf{Form 15}]{HR1998}): For every set $M$ there is a function $f$ such that for all $A \in M$, if $\vert A\vert > 1$ then $\emptyset \neq f(A) \subsetneq A$.

    \item The {\em Axiom of Multiple Choice}, $\mathsf{MC}$ (\cite[\textbf{Form 67}]{HR1998}): Every family $\mathcal{A}$ of non-empty sets has a {\em multiple choice function}, i.e., there is a function $f$ with domain $\mathcal{A}$ such that for every $A \in \mathcal{A}$, $\emptyset\neq f(A)\in [A]^{<\omega}$.
    \item $\mathsf{MC}^{\aleph_{0}}_{\aleph_{0}}$ (\cite[\textbf{Form 350}]{HR1998}): Every denumerable, i.e. countably infinite, family of denumerable sets has a multiple choice function. 
    
    \item $\mathsf{AC_{WO}}$ (\cite[\textbf{Form 60}]{HR1998}): Every set of non-empty, well-orderable sets has a choice function.
    \item $\mathsf{AC^{LO}}$ \cite[\textbf{Form 202}]{HR1998}: Every linearly ordered set of non-empty sets has a 
choice function.
    
    \item $\mathsf{PAC^{\aleph_{1}}_{fin}}$ (cf. \cite{Ban2}): Every $\aleph_{1}$-sized family $\mathcal{A}$ of non-empty finite sets has an $\aleph_{1}$-sized subfamily $\mathcal{B}$ with a choice function.

    \item  $\mathsf{AC_{\aleph_{0}}^{\aleph_{0}}}$ \cite[\textbf{Form 32A}]{HR1998}: Every denumerable family of denumerable sets has a choice function. We recall that $\mathsf{AC_{\aleph_{0}}^{\aleph_{0}}}$ is equivalent to $\mathsf{PAC_{\aleph_{0}}^{\aleph_{0}}}$ \cite[\textbf{Form 32B}]{HR1998} (Every denumerable family $\mathcal{A}$ of denumerable sets has an infinite subfamily $\mathcal{B}$ with a choice function).
    \item  $\mathsf{AC_{fin}^{\aleph_{0}}}$ \cite[\textbf{Form 10}]{HR1998}:  Every denumerable family of non-empty finite sets has a choice function. 
    \item $\mathsf{AC_{n}^{-}}$ for each $n\in\omega\backslash \{0,1\}$ \cite[\textbf{Form 342($n$)}]{HR1998}: Every infinite family $\mathcal{A}$ of $n$-element sets has a {\em partial choice function}, i.e., $\mathcal{A}$ has an infinite subfamily $\mathcal{B}$ with a choice function. 
    
    \item $\mathsf{WOAM}$ \cite[\textbf{Form 133}]{HR1998}: Every set is either well-orderable or has an amorphous subset.

    \item The {\em Principle of Dependent Choice}, $\mathsf{DC}$(\cite[\textbf{Form 43}]{HR1998}): If $S$ is a relation on a non-empty set $A$ and $(\forall x\in A)(\exists y\in A)(xSy)$ then there is a sequence $(a_{n})_{n\in\omega}$ of elements of $A$ such that $(\forall n\in \omega)(a_{n}S a_{n+1})$. 

    \item $\mathsf{DC_{\kappa}}$ 
    where $\alpha$ is the ordinal such that $\kappa=\aleph_{\alpha}$
    \cite[\textbf{Form 87($\alpha$)}]{HR1998}: Let $S$ be a non-empty set and let $R$ be a binary relation such that for every $\beta<\kappa$ and every $\beta$-sequence $s =(s_{\epsilon})_{\epsilon<\beta}$ of elements of $S$ there exists $y \in S$ such that $s R y$. Then there is a function $f : \kappa \rightarrow S$ such that for every $\beta < \kappa$, $(f\restriction \beta) R f(\beta)$. We note that $\mathsf{DC_{\aleph_{0}}}$ is a reformulation of $\mathsf{DC}$.
    
    \item $\mathsf{DF = F}$ \cite[\textbf{Form 9}]{HR1998}: Every Dedekind-finite set is finite.
    
    \item $\mathsf{W_{\aleph_{\alpha}}}$ (cf. \cite[\textbf{Chapter 8}]{Jec1973}): For every $X$, either $\vert X\vert \leq \aleph_{\alpha}$ or $\vert X\vert \geq \aleph_{\alpha}$. 
    \item \cite[\textbf{Form 233}]{HR1998}: If a field has an algebraic closure, then it is unique up to isomorphism. 
    \item  $\mathsf{WUT}$ \cite[\textbf{Form 231}]{HR1998}: The union of a well-orderable collection of well-orderable sets is well-orderable.
    
    \item $\mathsf{AC^{WO}_{fin}}$ \cite[\textbf{Form 122}]{HR1998}: Every well-ordered set of non-empty finite sets has a
    choice function. 
    
    \item The {\em Countable Union Theorem}, $\mathsf{CUT}$ \cite[\textbf{Form 31}]{HR1998}: The union of a countable family of countable sets is countable.
    
    \item {\em $\mathsf{CS}$}: Every poset without a maximal element has two disjoint cofinal subsets.
    
    \item {\em $\mathsf{CWF}$}: Every poset has a cofinal well-founded subset.
    \item The {\em Antichain Principle, $\mathsf{A}$}: Every poset has a maximal antichain.

    \item The {\em $n$-coloring theorem}, $\mathcal{P}_{n}$:
    If all finite subgraphs of a graph $G$ are $n$-colorable then $G$ is $n$-colorable.
    \item {\em Dilworth’s Theorem}, $\mathsf{DT}$: If $(P,\leq)$ is a poset of width $k$ for some $k\in \omega$, then $P$ can be partitioned into $k$ chains.
    
    \item {\em Ramsey’s Theorem}, $\mathsf{RT}$ \cite[\textbf{Form 17}]{HR1998}: For every infinite set $A$ and for every partition of the set $[A]^{2}$ into two sets $X$ and $Y$, there is an infinite subset $B \subseteq A$ such that either $[B]^2 \subseteq X$ or $[B]^2 \subseteq Y$.\footnote{Equivalently, for every infinite graph $G=(V,E)$ and for all $c : [V]^{2} \rightarrow 2$, there exists an infinite set
    $Y \subseteq V$ such that
    $[Y]^{2}$ is $c$-monochromatic.}
    
    \item The {\em Chain/Antichain Principle}, $\mathsf{CAC}$ \cite[\textbf{Form 217}]{HR1998}: Every infinite poset has an infinite chain or an infinite antichain.   
    
    \item {\em \L{}o\'{s}'s theorem}, $\mathsf{LT}$ (\cite[\textbf{Form 253}]{HR1998}): If $\mathcal{A}=\langle A, \mathcal{R}^{\mathcal{A}}\rangle$ is a non-trivial relational $\mathcal{L}$-structure over some language $\mathcal{L}$, and $\mathcal{U}$ be an ultrafilter on a non-empty set $I$, then the ultrapower $\mathcal{A}^{I}/\mathcal{U}$ and $\mathcal{A}$ are elementarily equivalent.  
\end{enumerate}
\end{defn}

\begin{defn}{\textbf{(A list of combinatorial statements).}}

\begin{enumerate}
    \item $\mathsf{EDM}$: If $G=(V_{G}, E_{G})$ is a graph such that $V_{G}$ is uncountable, then for all coloring $f:[V_{G}]^{2}\rightarrow \{0,1\}$ either there is an uncountable set monochromatic in color $0$, or there is a countably infinite set monochromatic in color 1.
    
    \item $\mathsf{EDM'}$: $\mathsf{EDM}$ restricted to graphs based on a well-ordered set of vertices.
    
    \item $\mathsf{CAC^{\aleph_{0}}}$:  Every poset such that all of its chains are finite and all of its antichains are countable is countable.
    
    \item $(\mathsf{CAC^{\aleph_{0}}})'$:  $\mathsf{CAC^{\aleph_{0}}}$ restricted to posets based on a well-ordered set of elements.
    
    \item $\mathsf{CAC_{1}^{\aleph_{0}}}$:  Every poset such that all of its antichains are finite and all of its chains are countable is countable.
    
    \item $(\mathsf{CAC_{1}^{\aleph_{0}}})'$:  $\mathsf{CAC_{1}^{\aleph_{0}}}$ restricted to posets based on a well-ordered set of elements.
    
    \item $\mathsf{CACT^{\aleph_{0}}}$:  $\mathsf{CAC^{\aleph_{0}}}$ restricted to $\omega$-trees.
    
    \item $(\mathsf{CACT^{\aleph_{0}}})'$:  $\mathsf{CAC^{\aleph_{0}}}$ restricted to $\omega$-trees based on a well-ordered set of elements.
    
    \item $\mathsf{CACT_{1}^{\aleph_{0}}}$:  $\mathsf{CAC_{1}^{\aleph_{0}}}$ restricted to $\omega$-trees.
    \item $(\mathsf{CACT_{1}^{\aleph_{0}}})'$:  $\mathsf{CAC_{1}^{\aleph_{0}}}$ restricted to $\omega$-trees based on a well-ordered set of elements.

    \item For a set $A$, Sym$(A)$ and  FSym$(A)$ denote the set of all permutations of $A$ and the set of all $\phi \in$ Sym$(A)$ such that $\{x \in A : \phi(x) \neq x\}$ is finite. For a set $A$ of size at least $\aleph_{\alpha}$, $\aleph_{\alpha}$Sym$(A)$ denote the set of all $\phi \in$ Sym$(A)$ such that $\{x \in A : \phi(x) \neq x\}$ has cardinality at most $\aleph_{\alpha}$  (cf. \cite[section 2]{Tac2019}).
\end{enumerate}
\end{defn}

\subsection{Permutation models and Mostowski’s intersection lemma.} We start with a model $M$ of $\mathsf{ZFA+AC}$ where $A$ is a set of atoms, $\mathcal{G}$ is a group of permutations of $A$ and $\mathcal{F}$ is a normal filter of subgroups of $\mathcal{G}$. The Fraenkel–Mostowski model, or the permutation model $\mathcal{N}$ with respect to $M$, $\mathcal{G}$  and $\mathcal{F}$ is defined by the equality:
\begin{center}

    $\mathcal{N} = \{x \in M : (\forall t \in \mathsf{TC}(\{x\}))(sym_{\mathcal{G}}(t) \in \mathcal{F})\}$
\end{center}
where for a set $x\in M$, $sym_{\mathcal {G}}(x) =\{g\in \mathcal {G} \mid g(x) = x\}$ and  $\mathsf{TC}(x)$ is the transitive closure of $x$ in $M$. If $\mathcal{I}\subseteq\mathcal{P}(A)$ is a normal ideal, then $\{$fix$_{\mathcal{G}}(E): E\in\mathcal{I}\}$ generates a normal filter (say $\mathcal{F}_{\mathcal{I}}$) over $\mathcal{G}$, where fix$_{\mathcal{G}}(E)=\{\phi \in \mathcal{G} : \forall y \in E (\phi(y) = y)\}$. Let $\mathcal{N}$ be the permutation model determined by $M$, $\mathcal{G}$, and
$\mathcal{F}_{\mathcal{I}}$. We recall that $\mathcal{N}$ is a model of $\mathsf{ZFA}$ (cf. \cite[Theorem 4.1, page 46]{Jec1973}). 
We say $E\in \mathcal{I}$ is a {\em support} of a set $\sigma\in \mathcal{N}$ if fix$_{\mathcal{G}}(E)\subseteq sym_{\mathcal{G}} (\sigma$). 
We recall some terminologies from \cite[sections 1,2] {Bru1985}. Let $\Delta(E)=\{\sigma:$ fix$_{\mathcal{G}}(E)\subseteq sym_{\mathcal{G}} (\sigma$)$\}$ if $E\in \mathcal{I}$. We say that $\mathcal{N}$ satisfies {\em Mostowski’s intersection lemma} if 
$\Delta(E \cap F) = \Delta(E) \cap \Delta(F)$ for every $E,F\in \mathcal{I}$.  
In this paper, 
\begin{itemize}
    \item We follow the labeling of the models from \cite{HR1998}. $\mathcal{N}_{1}$ is the basic Fraenkel model, $\mathcal{N}_{2}$ is the second Fraenkel model, $\mathcal{N}_{3}$ is the Mostowski linearly ordered model, and $\mathcal{N}_{41}$ is a variation of $\mathcal{N}_{3}$ (cf. \cite{HR1998}).
    
    \item Fix any $n\in \omega\backslash\{0,1\}$. We denote by $\mathcal{N}_{HT}^{1}(n)$ the permutation model constructed in \cite[Theorem 8]{HT2020}.
\end{itemize}

\begin{lem}
{\em An element $x$ of $\mathcal{N}$ is well-orderable in $\mathcal{N}$ if and only if {\em fix}$_{\mathcal{G}}(x)\in \mathcal{F}_{\mathcal{I}}$ {\em (cf. \cite[Equation (4.2), page 47]{Jec1973})}. Thus, an element $x$ of $\mathcal{N}$ with support $E$ is well-orderable in $\mathcal{N}$ if {\em fix}$_{\mathcal{G}}(E) \subseteq$ {\em fix}$_{\mathcal{G}}(x)$.
}
\end{lem}

We refer the reader to \cite[Note 103, pages 283–286]{HR1998} for the definition of the terms ``injective cardinality $\vert x \vert_{-}$ of $x$", ``injectively boundable statement" and ``boundable statement".
\begin{thm}{\textbf{(Pincus’ Transfer Theorem; cf. \cite[\textbf{Theorem 3A3}]{Pin1972})}}
{\em If $\Phi$ is a conjunction of injectively boundable statements which hold in the Fraenkel–Mostowski model $V_{0}$, then there is a $\mathsf{ZF}$ model $V \supset V_{0}$ with the same ordinals and cofinalities as $V_{0}$, where $\Phi$ holds.}
\end{thm}

\begin{lem}{\textbf{(Brunner; cf. \cite[\textbf{Lemma 4.1}}]{Bru1982})}
{\em Let $A$ be a set of atoms, $\mathcal{G}$ be a group of permutations of $A$ and the filter $\mathcal{F}$ of subgroups of $\mathcal{G}$ is generated by $\{${\em fix}$_{\mathcal{G}}(E):E\in [A]^{<\omega}\}$. Let $\mathcal{N}$ be the Fraenkel-Mostowski model determined by $A$, $\mathcal{G}$, and $\mathcal{F}$. If $\mathcal{N}$ satisfies Mostowski’s intersection lemma where $A$ is Dedekind-finite, then every set $x$ in $\mathcal{N}$ is either well-orderable or there exists an infinite subset of $A$, which embeds into $x$.}
\end{lem} 
\subsection{Diagram}  
We summarize the main results using the first definition of uncountable sets. Fix any $2\leq n\in\omega$. In Figure 1, known results are depicted with dashed arrows, new implications or non-implications in $\mathsf{ZF}$ are mentioned with simple black arrows, new non-implications in $\mathsf{ZFA}$ are mentioned with thick dotted black arrows.

\begin{figure}[!ht]
\begin{minipage}{\textwidth}
\begin{tikzpicture}[scale=6]
\draw (2.15,0.71) node[above] {$\mathsf{AC}$};
\draw[dashed, -triangle 60] (2.15,0.8) -- (2.15,1);
\draw (2.15,1) node[above] {$\mathsf{AC_{n}}$};
\draw[dashed, -triangle 60] (2.15,1.1) -- (1.85,1.26);
\draw[dashed, -triangle 60] (2.15,1.1) -- (2.45, 1.26);
\draw (1.85,1.25) node[above] {$\mathsf{AC_{n}^{-}}$};
\draw (2.45, 1.25) node[above] {``There are no amorphous sets"};

\draw[dashed, -triangle 60] (2.22,0.75) -- (2.37,0.75);
\draw (2.45,0.7) node[above] {$\mathsf{DC_{\aleph_{1}}}$};
\draw[dashed, -triangle 60] (2.45,0.8) -- (2.45,1);
\draw (2.45,1.02) node[above] {$\mathsf{DC}$};
\draw[dashed, -triangle 60] (2.45,1.1) -- (2.45,1.25);
\draw[-triangle 60] (2.52,0.75) -- (2.7,0.75);

\draw (2.9,0.7) node[above] {$\mathsf{EDM}$};

\draw[ultra thick, dotted, -triangle 60] (2.35,0.2) -- (2.85,0.7);
\draw[ultra thick, dotted, -triangle 60] (2.75,0.7) -- (2.25,0.2);
\draw (3.1,0.05) node[above] {$X$ if $X\in\{\mathsf{BPI},\mathsf{KW},\mathsf{AC_{WO}}, \mathcal{P}_{2}, \mathsf{DT},\mathsf{WOAM}, $ ``Antichain Principle $(\mathsf{A})$", $\mathsf{CS}$, $\mathsf{MC}\}$};

\draw[thick] (1.83,0.05) rectangle (4.38,0.18);

\draw(2.6,0.42) -- (2.6,0.48);
\draw(2.5,0.42) -- (2.5,0.48);

\draw[-triangle 60] (2.9,0.35) -- (2.9,0.7);
\draw (2.9,0.25) node[above] {$\mathsf{RT}$+$\mathsf{WOAM}$};

\draw[thick] (2.7,0.25) rectangle (3.09,0.35);

\draw[dashed, -triangle 60] (3.3,0.35) -- (3.3,0.47);
\draw[dashed, -triangle 60] (3.08,0.3) -- (3.2,0.3);
\draw (3.4,0.25) node[above] {$\mathsf{CAC}$+$\mathsf{WOAM}$};
\draw[thick] (3.2,0.25) rectangle (3.63,0.35);

\draw[-triangle 60] (3.38,0.35) [bend right=60] to (3.38,0.7);
\draw[ultra thick, dotted, -triangle 60] (2.9,0.8) -- (2.5, 1.25);
\draw (2.8,0.95) -- (2.76,0.9);
\draw[-triangle 60]  (2.45, 1.02) -- (2.71,0.8);

\draw (2.57,0.95) -- (2.53,0.9);

\draw[-triangle 60] (2.9,0.8) -- (2.9,1);
\draw[-triangle 60] (2.93,1) -- (3.3,0.8);
\draw (3.07,0.95) -- (3.03,0.9);

\draw[dashed,-triangle 60] (2.52,1.05) -- (2.8,1.05);
\draw (2.92,1.02) node[above] {$\mathsf{RT}$};
\draw[dashed,-triangle 60] (3,1.05) -- (3.25,1.05);

\draw[-triangle 60] (3.35,0.8) -- (3.35,1.02);
\draw[dashed,-triangle 60] (3.38,0.55) [bend right=60] to (3.38,1);

\draw (3.35,1.02) node[above] {$\mathsf{CAC}$};
\draw[-triangle 60] (3.1,0.76) -- (3.25,0.76);
\draw[dashed,-triangle 60] (3.45,1.05) -- (3.72,1.05);
\draw (4.13,1) node[above] {$\mathsf{AC_{fin}^{\aleph_{0}}}\sim$ ``K\H{o}nig's Lemma"};

\draw (3.2,1.25) node[above] {$\mathsf{CACT_{1}^{\aleph_{0}}}$};
\draw (3.6,1.26) node[above] {$\mathsf{CACT^{\aleph_{0}}}$};
\draw[ultra thick, dotted, -triangle 60] (3.55,1.27) -- (3.35,1.12);
\draw[ultra thick, dotted, -triangle 60] (3.15,1.27) -- (3.35,1.12);
\draw (3.35,1.45) node[above] {$\mathsf{AC_{fin}^{\aleph_{0}}}$};

\draw[-triangle 60] (3.35,1.45)-- (3.55,1.35);
\draw[-triangle 60] (3.35,1.45) -- (3.15,1.35);
\draw (3.23,1.25) -- (3.17,1.21);
\draw (3.46,1.25) -- (3.51,1.21);
\draw[ultra thick, dotted, -triangle 60] (3.25,0.79) -- (3.09,0.79);
\draw (3.18,0.77) -- (3.2,0.81);
\draw (3.35,0.71) node[above] {$\mathsf{CAC^{\aleph_{0}}}$};

\draw[-triangle 60] (3.1,0.76) -- (3.25,0.55);
\draw[ultra thick, dotted, -triangle 60] (3.23,0.54) -- (3.08,0.74);
\draw (3.18,0.57) -- (3.2,0.61);
\draw (3.35,0.45) node[above] {$\mathsf{CAC_{1}^{\aleph_{0}}}$};

\draw[ultra thick, dotted, -triangle 60] (3.3,0.55) -- (3.3,0.72);
\draw (3.28,0.60) -- (3.32,0.64);

\draw[-triangle 60] (3.45,0.76) -- (3.78,0.76);
\draw (3.9,0.7) node[above] {$\mathsf{AC_{\aleph_{0}}^{\aleph_{0}}}$};

\draw[dashed,-triangle 60] (3.45,0.5) -- (3.78,0.5);
\draw (3.9,0.45) node[above]{$\mathsf{PAC_{fin}^{\aleph_{1}}}$} ;
\end{tikzpicture}
\end{minipage}
\caption{\em In the above figure, we denote ``equivalent to" by $\sim$. We note that $(\forall n\geq 3)$ $\mathcal{P}_{n}\sim \mathsf{BPI}$.  
}
\end{figure}
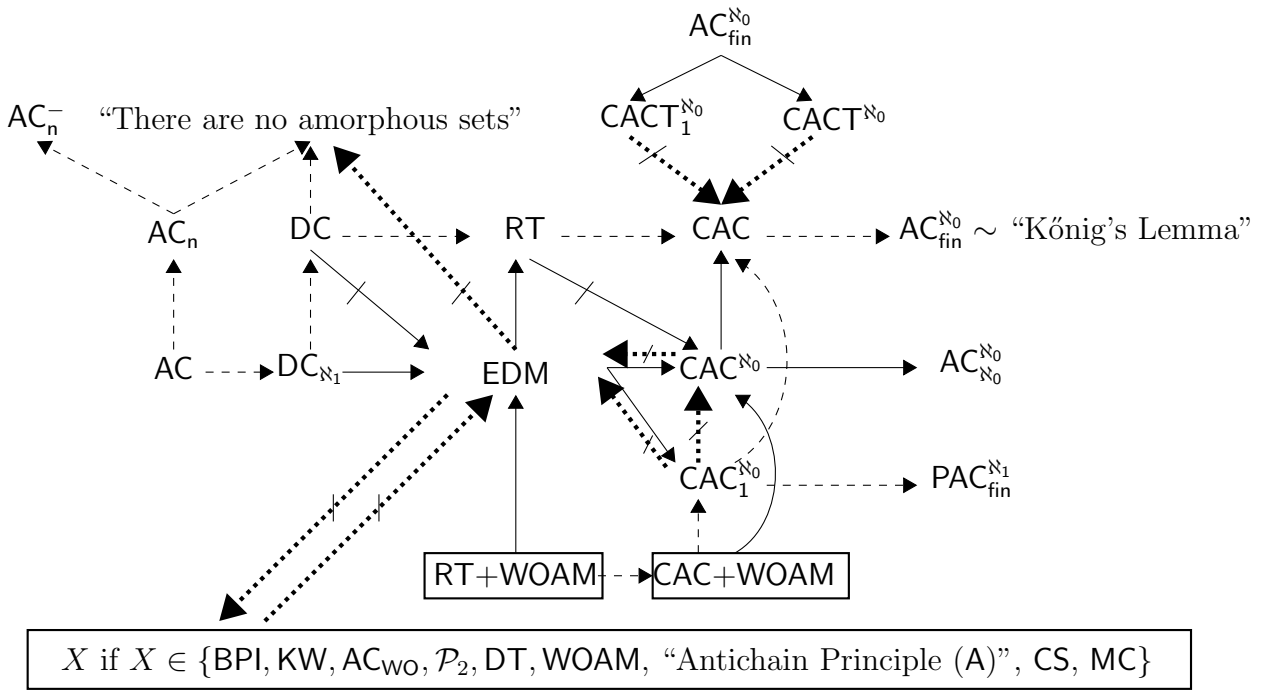
\section{Known and basic results}

\subsection{Known results}

\begin{fact}{($\mathsf{ZF}$)} {\em The following hold:
\begin{enumerate}
    \item $\mathsf{RT}$ holds for every infinite well-orderable set and if $\mathsf{RT}$ holds for an infinite set $Y$, then $\mathsf{RT}$ holds for any set $X \supseteq Y$ \cite[Theorem 1.7]{Tac2016a} and $\mathsf{DF=F}$ implies $\mathsf{RT}$ \cite{HR1998}.
    
    \item $\mathsf{WOAM}$ implies $\mathsf{CUT}$ \cite[Proposition 8(i)]{KTW2021}. So, $\mathsf{WOAM}$ implies ``$\aleph_{1}$ is regular".
    \item $\mathsf{CAC_{1}^{\aleph_{0}}}$ implies $\mathsf{CAC}$ \cite[Theorem 4(11)]{Tac2022} and $\mathsf{CAC}$ implies $\mathsf{AC_{fin}^{\aleph_{0}}}$ \cite[Lemma 4.4]{Tac2019a}.
    \item $\mathsf{WOAM}+\mathsf{CAC}$ implies $\mathsf{CAC^{\aleph_{0}}_{1}}$ \cite[Theorem 8(1)]{Tac2022}.
    \item $\mathsf{CAC^{\aleph_{0}}_{1}}$ implies $\mathsf{PAC^{\aleph_{1}}_{fin}}$ and $\mathsf{DC}$ does not imply  $\mathsf{CAC^{\aleph_{0}}_{1}}$ \cite[Theorem 4.5, Corollary 4.6]{Ban2}.
\end{enumerate}
}
\end{fact}
We recall the following result communicated to us by Tachtsis from \cite{Ban2}.

\begin{fact}(cf. \cite[Lemma 4.1, Corollary 4.2]{Ban2})
{\em $(\mathsf{CAC_{1}^{\aleph_{0}}})'$ holds in any permutation model.}
\end{fact}
\subsection{Basic propositions}
\begin{prop} 
{\em The following hold:
\begin{enumerate}
    \item ``$\aleph_{1}$ is regular" implies $\mathsf{EDM'}$ in $\mathsf{ZF}$.

    \item ``$\aleph_{1}$ is regular" implies $(\mathsf{CAC^{\aleph_{0}}})'$, $(\mathsf{CACT_{1}^{\aleph_{0}}})'$ as well as $(\mathsf{CACT^{\aleph_{0}}})'$ in $\mathsf{ZF}$.
    \item ``$\aleph_{1}$ is regular"$+$ $\mathsf{AC_{fin}^{\aleph_{0}}}$ implies $\mathsf{CACT_{1}^{\aleph_{0}}}$ and $\mathsf{CACT^{\aleph_{0}}}$ in $\mathsf{ZF}$.
    \item $X$ holds in any permutation model if $X\in \{\mathsf{EDM'},(\mathsf{CAC^{\aleph_{0}}})', (\mathsf{CACT_{1}^{\aleph_{0}}})',(\mathsf{CACT^{\aleph_{0}}})'\}$.
    
    \item  $\mathsf{AC_{fin}^{\aleph_{0}}}$ implies $\mathsf{CACT_{1}^{\aleph_{0}}}$ and $\mathsf{CACT^{\aleph_{0}}}$ in $\mathsf{ZF}$. 
\end{enumerate}
}
\end{prop}

\begin{proof}
(1). We modify the arguments due to Tachtsis from \cite[Lemma 4.1]{Ban2}. Let $G=(V_{G}, E_{G})$ be a graph based on a well-ordered set of vertices. Fix a well-ordering $\preceq$ of $V_{G}$. Let $f:[V_{G}]^{2}\rightarrow \{0,1\}$ be a coloring such that all sets monochromatic in color $0$ are countable and all sets monochromatic in color $1$ are finite.  
By way of contradiction, assume that $V_{G}$ is uncountable.
We construct an infinite set monochromatic in color $1$ in $G$ to obtain a contradiction. 
Since $V_{G}$ is well-ordered by $\preceq$, we can construct (via transfinite induction) a maximal set monochromatic in color $0$, $C_0$ say, without invoking any form of choice. Since $C_0$ is countable, it follows that $V_{G}-C_{0}$ is uncountable and for every vertex $v\in V_{G}-C_{0}$, there is $c\in C_0$ such that $f(\{v,c\})=1$. We write $V_{G}-C_{0} = \bigcup\{W_{p} : p\in C_{0}\}$, where $W_{p}=\{v\in V_{G}-C_{0}:f(\{v,p\})=1\}$. Since $V_{G}-C_{0}$ is uncountable and $C_0$ is countable, it follows by “$\aleph_{1}$ is regular” that $W_{p}$ is uncountable for some $p$ in $C_{0}$. Let $p_0$ be the least (with respect to $\preceq$) such vertex of $C_0$. Next, we construct a maximal set monochromatic in color $0$ in (the uncountable set) $W_{p_{0}}$, $C_1$ say, and let  (similarly to the above argument) $p_1$ be the least (with respect to $\preceq$) vertex of $C_1$ such that the set $W_{p_{1}}=\{v\in W_{p_{0}}-C_{1}: f(\{v,p_{1}\})=1\}$ is uncountable. Continuing this process step by step and noting that the process cannot stop at a finite stage, we obtain a countably infinite set of vertices $\{p_{n}: n\in\omega\}$ monochromatic in color $1$, contradicting the assumption that all sets monochromatic in color 1 are finite. Therefore, $V_{G}$ is countable.

(2--5). Follows from (1) and the fact that the statement “$\aleph_{1}$ is a regular cardinal” holds in every permutation model (cf. \cite[Corollary 1]{HKRST2001}) and $\mathsf{AC_{fin}^{\aleph_{0}}}$ is equivalent to ``Every $\omega$-tree is countable" in $\mathsf{ZF}$.
\end{proof}

\begin{prop}
($\mathsf{ZF}$)
{\em $\mathsf{CAC^{\aleph_{0}}}$ implies $\mathsf{AC^{\aleph_{0}}_{\aleph_{0}}}$.}
\end{prop}
\begin{proof}
Since $\mathsf{AC^{\aleph_{0}}_{\aleph_{0}}}$ is equivalent to its partial version $\mathsf{PAC^{\aleph_{0}}_{\aleph_{0}}}$ (cf. Definition 2.3), it suffices to show $\mathsf{PAC^{\aleph_{0}}_{\aleph_{0}}}$. Let $\mathcal{A} = \{A_{i}
:i\in \omega\}$ be a denumerable family of non-empty, denumerable sets. Without loss of generality, assume that $\mathcal{A}$ is disjoint. For the sake of contradiction, we assume that $\mathcal{A}$ has no partial choice function. Deﬁne a binary relation $\leq$ on $A =\bigcup \mathcal{A}$ as follows: for all $a,b \in A$, let $a\leq b$ if and only if $a = b$ or $a \in A_{n}$, $b \in A_{m}$ and $n < m$. 
Clearly, $\leq$ is a partial order on $A$. Since any two elements of $A$ are $\leq$-comparable if and only if they belong to distinct $A_{i}$'s, and $\mathcal{A}$ has no partial choice function, all chains in $(A,\leq)$ are finite. Next, if $C \subset A$ is an antichain in $(A,\leq)$, then $C \subseteq A_{i}$ for some $i \in \omega$. Thus, all antichains in $(A,\leq)$ are countable as $A_{i}$ is denumerable for all $i\in \omega$. By $\mathsf{CAC^{\aleph_{0}}}$, $A$ is countable (and hence well-orderable), contradicting $\mathcal{A}$'s having no partial choice function. 
\end{proof}

\begin{prop}{\em Let $A$ be a set of atoms. Let $\mathcal{G}$ be the group of permutations of A such that either each $\eta\in \mathcal{G}$ moves only finitely many atoms or there exists an $n\in\omega\backslash \{0,1\}$, such that for all $\eta\in \mathcal{G}$, $\eta^{n}=1_{A}$. Let $\mathcal{N}$ be the permutation model determined by $A$, $\mathcal{G}$, and a normal filter $\mathcal{F}$ of subgroups of $\mathcal{G}$. Then the following hold:
\begin{enumerate}
    \item The Antichain Principle $\mathsf{A}$ holds in $\mathcal{N}$.
    \item If $\mathsf{WUT}$ holds in $\mathcal{N}$, then both $\mathsf{CAC^{\aleph_{0}}}$ and $\mathsf{CAC_{1}^{\aleph_{0}}}$ hold in $\mathcal{N}$.
    \item If $\mathsf{AC^{WO}_{fin}}$ holds and $\mathsf{AC_{\aleph_{0}}^{\aleph_{0}}}$ fails in $\mathcal{N}$, then $\mathsf{CAC_{1}^{\aleph_{0}}}$ holds and $\mathsf{CAC^{\aleph_{0}}}$ fails in $\mathcal{N}$.
\end{enumerate}
}
\end{prop}

\begin{proof}
Let $(P,\leq)$ be a poset in $\mathcal{N}$. Then the subgroup $H = sym_{\mathcal{G}}((P, \leq))$
is an element of $\mathcal{F}$. 
Following the proof of \cite[Theorem 3]{Tac2022}, $Orb_{H}(p)= \{\phi(p) : \phi \in H\}$ is an antichain in $P$ for each $p\in P$ and $\mathcal{O}=\{Orb_{H}(p): p\in P\}$ is a well-ordered partition of $P$.

(1). In $\mathcal{N}$, $\mathsf{CS}$ and $\mathsf{CWF}$ hold following the methods of \cite[Theorem 3.26]{HST2016} and \cite[proof of Theorem 10 (ii)]{Tac2018}. In \cite{HST}, it has been established that $\mathsf{CWF}$ is equivalent to $\mathsf{A}$ in
$\mathsf{ZFA}$. Thus $\mathsf{A}$ holds in $\mathcal{N}$.

We can observe a different argument to show that $\mathsf{A}$ holds in $\mathcal{N}$ following the proof of \cite[Theorem 9.2(2)]{Jec1973} and the fact that $Orb_{H}(p)$ is an antichain in $P$ for each $p\in P$. 

(2). We show $\mathsf{CAC^{\aleph_{0}}}$ holds in $\mathcal{N}$. Let $(P,\leq)$ be a poset in $\mathcal{N}$ such that all chains in $P$ are finite and all antichains in $P$ are countable (and hence well-orderable). Now, $P$ can be written as a well-orderable disjoint union of antichains. Thus, $P$ is well-orderable in $\mathcal{N}$ since $\mathsf{WUT}$ holds in $\mathcal{N}$. So, we are done by Proposition 3.3(4). Similarly, $\mathsf{CAC_{1}^{\aleph_{0}}}$ holds in $\mathcal{N}$ by Fact 3.2.

(3). Follows from Proposition 3.4 and the arguments of (2).
\end{proof}

\begin{corr}
{\em $\mathsf{CS} + \mathsf{A}+ \mathsf{DF=F} + \mathsf{AC^{WO}_{fin}} + \mathsf{CAC_{1}^{\aleph_{0}}}+\mathsf{DT}$ does not imply $\mathsf{MC_{\aleph_{0}}^{\aleph_{0}}}$ in $\mathsf{ZFA}$. Consequently, $\mathsf{CAC_{1}^{\aleph_{0}}}$ does not imply $\mathsf{CAC^{\aleph_{0}}}$ in $\mathsf{ZFA}$.}
\end{corr}
\begin{proof}
Consider the permutation model (say $\mathcal{M}$) from \cite[proof of Theorem 5(4)]{Tac2022}. In  order  to  describe $\mathcal{M}$, we
start with a model $M$ of $\mathsf{ZFA + AC}$ with a countably infinite set $A$ of atoms, which is written as a disjoint union
$\bigcup\{B_{n} : n \in \omega\}$, where $\vert B_{n}\vert = \aleph_{0}$ for all $n \in \omega$.
For each $n \in \omega$, let $\mathcal{G}_{n}$ be the group of all even permutations of $B_{n}$ which move only finitely
many elements of $B_{n}$. Let $\mathcal{G}$ be the weak direct product of the $\mathcal{G}_{n}$’s for $n \in \omega$.
Consequently, {\em every permutation of $A$ in $\mathcal{G}$ moves only finitely many atoms}.
Let $\mathcal{I}$ be the normal ideal of subsets of $A$ generated by all finite unions of $B_{n}$. Let $\mathcal{F}$ be the normal filter on $\mathcal{G}$ generated by $\{$fix$_{\mathcal{G}}(E)$, $E \in \mathcal{I}\}$ and $\mathcal{M}$ be the permutation model determined by $M$, $\mathcal{G}$, and $\mathcal{F}$. In $\mathcal{M}$, $\mathsf{AC^{WO}_{fin}}$, $\mathsf{DF=F}$, and $\mathsf{CAC_{1}^{\aleph_{0}}}$ hold whereas $\mathsf{MC_{\aleph_{0}}^{\aleph_{0}}}$ fails, and thus $\mathsf{AC_{\aleph_{0}}^{\aleph_{0}}}$ fails (cf. \cite[proof of Theorem 5(4)]{Tac2022}). Since $\mathsf{AC^{WO}_{fin}}$ holds in $\mathcal{N}$, and $\{a \in A : g(a)\not = a\}$ is finite for any $g\in \mathcal{G}$, $\mathsf{DT}$ holds in $\mathcal{N}$ following the arguments of \cite[Theorem 3.4]{Tac2019c}  where Tachtsis proved that $\mathsf{DT}$ holds in L\'{e}vy’s permutation
model (labeled as Model $\mathcal{N}_{6}$ in \cite{HR1998}).
The rest follows from Proposition 3.5 and the fact that if $g\in \mathcal{G}$, then $\{a \in A: g(a)\not = a\}$ is finite.
\end{proof}

\begin{prop}{($\mathsf{ZF}+\mathsf{WOAM}$)}
{\em $\mathsf{RT}$ holds for any locally countable connected graph $H=(V_{H},E_{H})$.}
\end{prop}

\begin{proof}
If $V_{H}$ is well-orderable, the conclusion follows from Fact 3.1(1). Otherwise, by $\mathsf{WOAM}$, there exists an amorphous subset $V_{G}\subseteq V_{H}$. Fix some $r \in V_{G}$. Let $V_{0}=\{r\}$. For each $n \in \omega\backslash \{0\}$, define $V_{n} = \{v \in V_{G} : d_{G}(r, v) = n\}$ where ``$d_{G}(r, v) = n$" means there are $n$ edges in the shortest path joining $r$ and $v$. By connectedness of $G$, $V_{G} = \bigcup_{n\in \omega}V_{n}$. Since $V_{G}$ is amorphous, there is at most one $t\in\omega\backslash \{0\}$, such that $V_{t}$ is infinite. 
As $V_{G}$ is amorphous, the power set $\mathcal{P}(V_{G})$ of $V_{G}$ is Dedekind-finite, and thus, for
some $n_{0} \in \omega \backslash \{0\}$, $V_{n} = \emptyset$ for all $n \geq n_{0}$. As $V_{G}$ is amorphous (and thus also infinite) and
$V_{G} =\bigcup_{n\in \omega}V_{n}$ is a disjoint union, there exists exactly one $t < n_{0}$ such that $V_{t}$ is infinite. Then $V_{t}$ is countably infinite since $\vert V_{t-1}\vert<\omega$, $G$ is locally countable, and the union of a finite family of countable sets is countable in $\mathsf{ZF}$. As $V_{t}$ is a countably infinite subset of the amorphous set $V_{G}$, which is impossible (since $V_{G}$ is Dedekind-finite, being amorphous), we arrive at a contradiction.
\end{proof}
\section{Erd\H{o}s--Dushnik--Miller theorem and its variants}

\begin{thm}$(\mathsf{ZF}$)
{\em The following hold:

\begin{enumerate}
   \item $\mathsf{DC_{\aleph_{1}}}$ implies  $\mathsf{EDM}$. In particular,  $\mathsf{W_{\aleph_{1}}}$ {\em +} $``\aleph_{1}$ is regular" implies $\mathsf{EDM}$.
   
   \item If $X\in \{\mathsf{CAC_{1}^{\aleph_{0}}}, \mathsf{CAC^{\aleph_{0}}},\mathsf{PAC^{\aleph_{1}}_{fin}},\mathsf{AC^{\aleph_{0}}_{\aleph_{0}}}, \mathsf{CAC}, \mathsf{AC_{fin}^{\aleph_{0}}}\}$ then $\mathsf{EDM}$ implies $X$. So, $\mathsf{DC}$ does not imply  $\mathsf{EDM}$.
   
   \item $\mathsf{EDM}$ implies $\mathsf{RT}$. 
    
   \item $\mathsf{DF=F}$ does not imply $\mathsf{CAC^{\aleph_{0}}}$. Consequently, $\mathsf{RT}$ does not imply $\mathsf{CAC^{\aleph_{0}}}$ and $\mathsf{EDM}$.
   
\end{enumerate}
}
\end{thm}

\begin{proof}
(1). Following Proposition 3.3 and the arguments of \cite[Theorem 9(1,2)]{Tac2022}, we can see that $\mathsf{DC_{\aleph_{1}}}$ implies  $\mathsf{EDM}$ in $\mathsf{ZF}$.
In particular, let $G=(V_{G}, E_{G})$ be a graph and $f:[V_{G}]^{2}\rightarrow \{0,1\}$ be a coloring such that all sets monochromatic in color $0$ are countable and all sets monochromatic in color $1$ are finite. By $\mathsf{W_{\aleph_{1}}}$, $\aleph_{1} \leq \vert V_{G}\vert$ or $\vert V_{G}\vert\leq \aleph_{1}$. For the second case, $V_{G}$ is well-orderable, and we are done by Proposition 3.3 since $\mathsf{DC_{\aleph_{1}}}$ implies $\mathsf{W_{\aleph_{1}}}$ + ``$\aleph_{1}$ is regular". Otherwise, $V_{G}$ has a subset $H$ with cardinality $\aleph_{1}$. Since $H$ is well-orderable, it is countable by the arguments of Proposition 3.3; a contradiction. 

(2). We prove $\mathsf{EDM}$ implies $\mathsf{CAC^{\aleph_{0}}}$. Let $(P,\leq)$ be a poset satisfying the hypotheses of $\mathsf{CAC^{\aleph_{0}}}$. Assume that $P$ is uncountable. Let $G=(V_{G},E_{G})$ be a complete graph such that $V_{G}=P$ and $f:[V_{G}]^{2}\rightarrow \{0,1\}$ be a coloring such that
$f\{x,y\}=1$ if $x\leq y$ or $y\leq x$, and $f\{x,y\}=0$ otherwise. 
By $\mathsf{EDM}$, either there is an uncountable set monochromatic in color $0$ (which is an antichain in $(P,\leq)$) or there is a countably infinite set monochromatic in color $1$ (which is a chain in $(P,\leq)$), a contradiction. Similarly, we can prove $\mathsf{EDM}$ implies $\mathsf{CAC_{1}^{\aleph_{0}}}$. The rest follows from Proposition 3.4 and Fact 3.1.

(3). Let $A$ be an infinite set such that $\mathsf{RT}$ fails for $A$. 
Let $\{X,Y\}$ be a partition of $[A]^{2}$ such that there are no infinite subsets $B$ of $A$ for which 
either $[B]^2 \subseteq X$ or $[B]^2 \subseteq Y$.
Let $G=(V_{G},E_{G})$ be a complete graph such that $V_{G}=A$ and $f:[V_{G}]^{2}\rightarrow \{0,1\}$ be a coloring such that
$f\{x,y\}=1$ if $\{x, y\}\in X$ and $f\{x,y\}=0$ if $\{x, y\}\in Y$.
By assumption, all sets monochromatic in color $i$ are finite for each $i\in\{0,1\}$. By $\mathsf{EDM}$, $\vert V_{G}\vert \leq \aleph_{0}$ (since we are using the first definition of uncountable sets), and thus $V_{G}=A$ is well-orderable. The contradiction follows from the fact that $\mathsf{RT}$ holds for $A$ in $\mathsf{ZF}$ (cf. Fact 3.1). 

(4). Consider the model $\mathcal{N}_{41}$ from \cite{HR1998}. 
We start with a model $M$ of $\mathsf{ZFA + AC}$ where $A =\bigcup \{B_{n}:n \in \omega\}$ is a disjoint union, where each $B_{n}$ is countably infinite  and for each $n \in\omega$, 
$(B_{n}, \leq_{n}) \cong (\mathbb{Q}, \leq)$ (i.e., ordered like the rationals by $\leq_{n}$). Let $\mathcal{G}$ be the group of all permutations on $A$ such that for all $n \in \omega$, and all $\phi\in \mathcal{G}$, $\phi$ is an order automorphism of $(B_{n}, \leq_{n})$. Let $\mathcal{I}$ be the normal ideal of subsets of $A$ which is generated by finite unions of $B_{n}$’s and $\mathcal{F}$ be the normal filter on $\mathcal{G}$ generated by  $\{$fix$_{\mathcal{G}}(E)$, $E \in \mathcal{I}\}$. Let $\mathcal{N}_{41}$ be the Fraenkel–Mostowski model determined by $M$, $\mathcal{G}$, and $\mathcal{F}$.

In $\mathcal{N}_{41}$, $\mathsf{DF=F}$ holds and $\mathsf{AC^{\aleph_{0}}_{\aleph_{0}}}$ fails (cf. \cite[Theorem 4]{Tac2019b}, \cite[Note 112]{HR1998}). Pincus \cite{Pin1972} showed that $\mathsf{DF = F}$ is  equivalent to 
\begin{center}
$\forall x(\vert x\vert_{-} \leq \omega\rightarrow$ $(\mathsf{Df}(x)\rightarrow\psi(x)$)), 
\end{center}
where $\mathsf{Df}(x)\leftrightarrow \neg(\exists y)(y \subseteq x \land \vert y\vert = \omega)$ and $\psi(x)=$``$x$ is finite" (we note that $\psi(x)$ is a boundable formula). Thus, $\mathsf{DF = F}$ is injectively boundable.
Furthermore, $\neg \mathsf{AC^{\aleph_{0}}_{\aleph_{0}}}$ is  boundable, and hence injectively boundable.
Since $\phi$ = ``$\mathsf{DF=F}$ $\land$ $\neg\mathsf{AC^{\aleph_{0}}_{\aleph_{0}}}"$ is a conjunction of injectively boundable statements, which has a $\mathsf{ZFA}$ model, it follows
from Theorem 2.6 that $\phi$ has a $\mathsf{ZF}$ model.
By Proposition 3.4, we can see that $\mathsf{DF=F}$ (and thus $\mathsf{RT}$) does not imply $\mathsf{CAC^{\aleph_{0}}}$ in $\mathsf{ZF}$. 
\end{proof}

\begin{thm}
{\em The following hold:
\begin{enumerate}
    \item $\mathsf{WOAM + RT}$ implies $\mathsf{EDM}$ and $\mathsf{WOAM + CAC}$ implies $\mathsf{CAC^{\aleph_{0}}}$ in $\mathsf{ZF}$.  In particular, $\mathsf{EDM}$ does not imply ``There are no amorphous sets" in $\mathsf{ZFA}$.
    \item Let $A$ be a set of atoms, $\mathcal{G}$ be any group of permutations of $A$, and the filter $\mathcal{F}$ of subgroups of $\mathcal{G}$ is generated by $\{${\em fix}$_{\mathcal{G}}(E):E\in [A]^{<\omega}\}$. Let $\mathcal{N}$ be the permutation model determined by $A$, $\mathcal{G}$, and $\mathcal{F}$. If $\mathcal{N}$ satisfies Mostowski’s intersection lemma where $A$ is Dedekind-finite, and $\mathsf{RT}$ holds in $\mathcal{N}$, then $\mathsf{EDM}$ holds in $\mathcal{N}$.
    
    \item $\mathsf{EDM}$ holds in $\mathcal{N}_{3}$. Consequently, $\mathsf{EDM}$ implies none of $\mathsf{WOAM}$, $\mathsf{CS}$, and $\mathsf{A}$(Antichain Principle) in $\mathsf{ZFA}$.
\end{enumerate}
}
\end{thm}

\begin{proof}
(1). Assume that $\mathsf{WOAM}$ + $\mathsf{RT}$ is true. 
Let $G=(V_{G}, E_{G})$ be a graph and $f:[V_{G}]^{2}\rightarrow \{0,1\}$ be a coloring such that all sets monochromatic in color $0$ are countable, and all sets monochromatic in color $1$ are finite. If $V_{G}$ is well-orderable then we are done by  Proposition 3.3 and the fact that $\mathsf{WOAM}$ implies `$\aleph_{1}$ is regular' in $\mathsf{ZF}$ (cf. Fact 3.1(2)). 
Assume $V_{G}$ is not well-orderable. By $\mathsf{WOAM}$, $V_{G}$ has an amorphous subset, say $A$.  
Define the following partition of $[A]^{2}$:
\begin{center}
    $X=\{\{a,b\}\in [A]^{2}: f\{a,b\}=0\}$, $Y=\{\{a,b\}\in [A]^{2}: f\{a,b\}=1\}$.
\end{center}
Since $(A, E_{G}\restriction A)$ is an infinite graph where all sets monochromatic in color $1$ are finite, there is no infinite subset $B' \subseteq A$ such that $[B']^2 \subseteq Y$. By $\mathsf{RT}$, there is an infinite subset $B \subseteq A$ such that $[B]^2 \subseteq X$. So $(A, E_{G}\restriction A)$ has an infinite set monochromatic in color $0$, say $C$.
By assumption, $C$ is a countably infinite subset of $A$. This contradicts the fact
that $A$ is amorphous.

Similarly, $\mathsf{WOAM + CAC}$ implies $\mathsf{CAC^{\aleph_{0}}}$ in $\mathsf{ZF}$ following Proposition 3.3. The rest follows from the fact that $\mathsf{WOAM + RT}$ + ``There exists an amorphous set” is true in the basic Fraenkel model $\mathcal{N}_{1}$ (cf. \cite{Bla1977, HR1998}).

(2). In $\mathcal{N}$, assume $G=(V_{G}, E_{G})$ and $f:[V_{G}]^{2}\rightarrow \{0,1\}$ as in the proof of (1). If $V_{G}$ is well-orderable then we are done by  Proposition 3.3. Otherwise, by Lemma 2.7, there exists a bijection from an infinite subset $A'$ of $A$ onto some $H\subset V_{G}$ under the given assumptions.
Define the following partition of $[H]^{2}$ as in the proof of (1):
$X=\{\{a,b\}\in [H]^{2}: f\{a,b\}=0\}$, $Y=\{\{a,b\}\in [H]^{2}: f\{a,b\}=1\}$.
By $\mathsf{RT}$ and following the arguments of (1), there is a countably infinite subset $C$ of $H$. Thus $A'$ is Dedekind-infinite in $\mathcal{N}$ since $\vert H\vert=\vert A'\vert$. Consequently, the set $A$ of atoms is Dedekind-infinite in $\mathcal{N}$, which contradicts the fact that $A$ is a Dedekind-finite set in $\mathcal{N}$.

(3). We recall the definition of $\mathcal{N}_{3}$  from \cite{HR1998}. We start with a model $M$ of $\mathsf{ZFA + AC}$ with a countably infinite set $A$ of atoms
using an ordering $\leq$ on $A$ chosen so that $(A, \leq)$ is order-isomorphic to the set $\mathbb{Q}$ of the rational numbers with the usual ordering. Let $\mathcal{G}$ be the group of all
order automorphisms of $(A, \leq)$ and $\mathcal{F}$ be the normal filter on $\mathcal{G}$ generated
by the subgroups $\{\text{fix}_{\mathcal{G}}(E), E \in [A]^{<\omega}\}$. Let $\mathcal{N}_{3}$ be the Fraenkel--Mostowski model
determined by $M$, $\mathcal{G}$, and $\mathcal{F}$. The rest follows from (2), and the following known facts about $\mathcal{N}_{3}$:

\begin{enumerate}[(i)]
    \item $\mathcal{N}_{3}$ satisfies Mostowski’s intersection lemma (cf. \cite{Jec1973}). 
    
    \item $\mathsf{RT}$ is true in $\mathcal{N}_{3}$ (cf. \cite[Theorem 2.4]{Tac2016a}).
    
    \item The set of atoms $A$ is a Dedekind-finite set in $\mathcal{N}_{3}$.
    \item $\mathsf{WOAM}$ fails in $\mathcal{N}_{3}$ (cf. \cite{HR1998}).
    \item $\mathsf{CS}$ and $\mathsf{LW}$ (Every linearly ordered set can be well ordered) fail in $\mathcal{N}_{3}$ \cite[Theorem 7]{Tac2022} and $\mathsf{A}$ implies $\mathsf{LW}$ in $\mathsf{ZFA}$ \cite[Theorem 9.1]{Jec1973}.
\end{enumerate}
\end{proof}

\begin{prop}{($\mathsf{ZF}$)}
{\em The statements $\mathsf{EDM^{n}}$ (``If $G=(V_{G}, E_{G})$ is a graph such that $V_{G}$ is uncountable, then for all coloring $f:[V_{G}]^{2}\rightarrow n$ either there is an uncountable set $X_{1}\subseteq V_{G}$  monochromatic in color $i_{1}$  or there is a countably infinite set $X_{2}\subseteq V_{G}$ monochromatic in color $i_{2}$ for some $i_{1}, i_{2}\in n$ such that $i_{1}\neq i_{2}$")
are equivalent for all integers
$n \geq 2$. Moreover, $\mathsf{EDM^{n}}$ implies $\mathsf{RT}$ for all
$n \in\omega\backslash\{0,1\}$.} 
\end{prop}
\begin{proof}
Since any $f:[V_{G}]^{2}\rightarrow n$ maps $[V_{G}]^{2}$ to $m$ if $m>n\geq 2$, $\mathsf{EDM^{m}}$ implies $\mathsf{EDM^{n}}$ under these circumstances. 
We prove that $\mathsf{EDM^{n}}$ implies $\mathsf{EDM^{n+1}}$ for $n\geq 2$. The rest follows by mathematical induction.
Let $f:[V_{G}]^{2}\rightarrow n+1$ be a coloring where $V_{G}$ is uncountable. Let $f_{1} : [V_{G}]^2 \rightarrow n$  be given by $f_{1}(A) = min(f(A), n-1)$.
By $\mathsf{EDM^{n}}$, either there is an uncountable set $X_{1}\subseteq V_{G}$ that is $f_{1}$-monochromatic in color $i_{1}$, or there is a countably infinite set $X_{2}\subseteq V_{G}$ that is $f_{1}$-monochromatic in color $i_{2}$ for some $i_{1}, i_{2}\in n$ such that $i_{1}\neq i_{2}$. 
Fix $k\in\{1,2\}$. If $i_{k}\leq n-2$, then $f[X_{k}]^{2}=i_{k}\in n+1$. 

\textbf{Case (i):} Let $i_{1}=n-1$, then $f(A)\in \{n,n-1\}$ for all  $A\in [X_{1}]^2$. Define $f_{2}:[X_{1}]^2\rightarrow 2$ by $f_{2}(A)=n-f(A)$. By $\mathsf{EDM^{2}}$ (which follows from $\mathsf{EDM^{n}}$),  
either there is an uncountable set $Y_{1}\subseteq X_{1}$ that is $f_{2}$-monochromatic in color $j$, or there is a countably infinite set $Y_{2}\subseteq X_{1}$ that is $f_{2}$-monochromatic in color $1-j$ where $j \in \{0, 1\}$. 
Thus either $f[Y_{1}]^{2}=n-j\in n+1$ or $f[Y_{2}]^{2}=n-1+j\in n+1$.

\textbf{Case (ii):} Let $i_{2}=n-1$, then $f(A)\in \{n,n-1\}$ for all  $A\in [X_{2}]^2$. Define the following partition of $[X_{2}]^{2}$:
\begin{center}
    $X=\{\{a,b\}\in [X_{2}]^{2}: f\{a,b\}=n\}$, $Y=\{\{a,b\}\in [X_{2}]^{2}: f\{a,b\}=n-1\}$.
\end{center}
By $\mathsf{RT}$ (which follows from $\mathsf{EDM^{2}}$, see the arguments of Theorem 4.1(3)), there is a countably infinite subset $B \subseteq X_{2}$ such that either $[B]^2 \subseteq X$ or $[B]^2 \subseteq Y$. Thus either $f[B]^{2}=n\in n+1$ or $f[B]^{2}=n-1\in n+1$.

This completes the proof of the first assertion. Following Theorem 4.1(3), $\mathsf{EDM^{2}}$ implies $\mathsf{RT}$. Thus, $\mathsf{EDM^{n}}$ implies $\mathsf{RT}$ for all
$n \in\omega\backslash\{0,1\}$.
\end{proof}

\begin{thm}
{\em Fix any $n\in\omega\backslash\{0,1\}$ and $X\in\{\mathsf{CAC^{\aleph_{0}}},
\mathsf{CAC_{1}^{\aleph_{0}}}, \mathsf{CS}\}$. There is a model $\mathcal{M}$ of $\mathsf{ZFA}$ where $X$ holds but $\mathsf{AC_{n}^{-}}$ and the statement ``there are no amorphous sets" fail.
Moreover, the following hold in $\mathcal{M}$:
\begin{enumerate}
    \item $\neg\mathsf{EDM}$ and $\neg\mathsf{EDM^{k}}$ for each $k\geq 2$.
    \item Antichain Principle $\mathsf{A}$.
    
\end{enumerate}
}
\end{thm}

\begin{proof} 
Consider the permutation model $\mathcal{N}_{HT}^{1}(n)$ constructed by Halbeisen--Tachtsis in the proof of \cite[Theorem 8]{HT2020} where $\mathsf{AC_{n}^{-}}$ fails. Let $M$ be a model of  $\mathsf{ZFA+AC}$ where $A$ is a countably infinite set of atoms written as a disjoint union $\bigcup\{A_{i}:i\in \omega\}$ where for all $i\in \omega$, $A_{i}=\{a_{i_{1}}, a_{i_{2}},..., a_{i_{n}}\}$ and $\vert A_{i}\vert = n$. The group $\mathcal{G}$ is defined in \cite{HT2020} in a way so that {\em if $\eta\in \mathcal{G}$, then $\eta$ only moves finitely many atoms} and for all $i\in \omega$, $\eta(A_{i})=A_{k}$ for some $k\in \omega$. Let $\mathcal{F}$ be the normal filter generated by $\{$fix$_{\mathcal{G}}(E): E\in [A]^{<\omega}\}$ where $\mathcal{I} = [A]^{<\omega}$ is the normal ideal. 
The model $\mathcal{N}_{HT}^{1}(n)$ is
the permutation model determined by $M$, $\mathcal{G}$, and $\mathcal{F}$. The set of atoms $A$ is amorphous in $\mathcal{N}_{HT}^{1}(n)$ (cf. the proof of \cite[Theorem 8]{HT2020}). 
Banerjee \cite{Ban2, BG2021} observed that $\mathsf{CAC^{\aleph_{0}}}$,   $\mathsf{CS}$, and $\mathsf{CAC_{1}^{\aleph_{0}}}$ hold in $\mathcal{N}_{HT}^{1}(n)$ (cf. Proposition 3.5 as well).

(1). Follows from Theorem 4.1(3), Proposition 4.3, and the fact that $\mathsf{RT}$ fails in $\mathcal{N}_{HT}^{1}(n)$ (cf. \cite{Tac2016a}).

(2). Follows from Proposition 3.5 since if $\eta\in \mathcal{G}$, then $\eta$ only moves finitely many atoms.
\end{proof}

\begin{remark}
The referee communicated that neither $\mathsf{CACT^{\aleph_{0}}}$
nor $\mathsf{CACT_{1}^{\aleph_{0}}}$
is provable in $\mathsf{ZF}$. Consider the second Fraenkel model $\mathcal{N}_{2}$ of \cite{HR1998}, in which the set $A$ of atoms is a countable disjoint union of pairs so that $A =\bigcup \{A_{n}: n \in \omega\}$ where $\vert A_{n}\vert = 2$ for all $n \in \omega$ and, for all $n, m \in \omega$ with $n \not= m$,
$A_{n} \cap A_{m} = \emptyset$; $\mathcal{G}$ is the group of all permutations of $A$ which fix $A_{n}$ for every $n \in \omega$, and $\mathcal{F}$ is
the normal filter on $\mathcal{G}$ generated by the subgroups fix$_{\mathcal{G}}(E), E \in [A]^{<\omega}$. Let

\begin{center}
$P=\{\emptyset\}\bigcup \{f: f$ is a choice function for $\{A_{i}
: i \leq n\}$ for some $n \in \omega\}$.    
\end{center}

Define a partial order $\leq$ on $P$ by stipulating, for all $p, q \in P$, $p \leq q$ if and only if $p \subseteq q$. It
is clear that $(P, \leq) \in \mathcal{N}_{2}$ (because $Sym_{\mathcal{G}}((P, \leq))=\mathcal{G}\in \mathcal{F}$). Furthermore, $(P, \leq)$ is an $\omega$-tree
(with $\emptyset$ as its root), whose chains are all finite since the family $\mathcal{A} = \{A_{n} : n \in \omega\}$, which is
countable in $\mathcal{N}_{2}$, does not have a partial choice function in $\mathcal{N}_{2}$. 
Following the proof of \cite[Theorem 2.11]{Tac2016b}, we can see that
all antichains in $P$ are finite in $\mathcal{N}_{2}$. For the sake of contradiction, assume that $U$ is an infinite antichain in $P$. Let $E\in [A]^{<\omega}$ be a support of $U$. Without loss of
generality assume that $E = \bigcup_{i\leq k} A_{k}$ for some $k \in \omega$. Let 
\begin{center}
$V = \{p \in P : dom(p) = \{A_{0},..., A_{k}\}\land\exists f \in
U(p\subsetneq f)\}$.
\end{center}
Since $U$ is infinite and $\prod_{i\leq k}A_{i}$ is finite, we have that $V  \neq\emptyset$ and there is an element $p_{0}\in V$ such that $Y := \{f \in U : p_{0} \subsetneq f \}$ is infinite. This yields the existence of at least two elements $f$ and $g$ of $Y$ such that $dom(f) \subsetneq dom(g)$.
Let 
\begin{center}
$M = \{m \in \omega : A_{m} \in dom(f) \cap dom(g)$ and $f(A_{m}) 
\neq g(A_{m})\}$. 
\end{center}
Then for all $m \in M$, we have $m > k$.
Let $\phi$ be the permutation of $A$ which swaps $f(A_{m})$ and $g(A_{m})$ for all $m \in M$, and fixes all the other atoms. Clearly, $\phi \in $ fix$_{\mathcal{G}}(E)$. Since $E$ is a support of $U$, we have $\phi(U) = U$. Thus $\phi(f)\in U$. However, $\phi(f) = g\restriction dom(f)$,
so $\phi(f) \subsetneq g$. This contradicts the fact that $U$ is an antichain in $P$. Thus, every antichain in $P$ is finite.

However, $P$ is not countable in $\mathcal{N}_{2}$, and thus $\mathsf{CACT^{\aleph_{0}}}$
and $\mathsf{CACT_{1}^{\aleph_{0}}}$ are both false in $\mathcal{N}_{2}$.
Now, since $\neg\mathsf{CACT^{\aleph_{0}}}$
and $\neg\mathsf{CACT_{1}^{\aleph_{0}}}$ are boundable and have a permutation model, it follows
from the Jech–Sochor First Embedding Theorem (see \cite[Theorem 6.1]{Jec1973}) that they have a
symmetric $\mathsf{ZF}$-model.
\end{remark}

\section{\L{}o\'{s}'s theorem, and the uniqueness of algebraic closures}
We recall a fact that we need in order to prove Theorem 5.2.
\begin{fact}
    If $\mathcal{K}$ is an algebraically closed field, if $\pi$ is a non-trivial automorphism of $\mathcal{K}$ satisfying $\pi^{2}= 1_{\mathcal{K}}$, and if $i=\sqrt{-1} \in\mathcal{K}$, then $\pi(i) = -i \neq i$ (cf. \cite[Note 41]{HR1998}).     
\end{fact}

\begin{thm}
{\em $(\mathsf{ZFA})$ 
The following hold:
\begin{enumerate}
    \item {\em $\mathsf{AC^{LO}}$ + \textbf{Form 233}} neither implies $\mathsf{LT}$ nor implies $\mathsf{W_{\aleph_{1}}}$.
    \item {\em $\mathsf{AC^{LO}}$ + \textbf{Form 233} +  $\mathsf{EDM}$} neither implies $\mathsf{LT}$ nor implies $\mathsf{W_{\aleph_{2}}}$.
\end{enumerate}
}
\end{thm}

\begin{proof}
(1). We consider the permutation model $\mathcal{N}$ given in the proof of \cite[Theorem 4.7]{Tac2019a} where $\mathsf{AC^{LO}}$ holds and $\mathsf{LT}$ fails. We start with a model $M$ of $\mathsf{ZFA + AC}$ with an $\aleph_{1}$-sized set $A$ of atoms so that $A = \bigcup \{A_{i} : i < \aleph_{1}\}$ where $\vert A_{i}\vert=\aleph_{0}$ for all $i < \aleph_{1}$, and $A_{i}\cap A_{j} = \emptyset$ for all
$i, j < \aleph_{1}$ with $i \not= j$. Let $\mathcal{G}$ be the group of all permutations $\phi$ of $A$ such that $\forall i < \aleph_{1}, \exists j < \aleph_{1}, \phi(A_{i}) = A_{j}$,
and $\phi$ moves only $\aleph_{0}$ atoms. Let $\mathcal{F}$ be the normal filter of subgroups of $\mathcal{G}$ generated by fix$_{\mathcal{G}}(E)$, where $E = \bigcup\{A_{i} : i \in \mathcal{I}\}, \mathcal{I} \in [\aleph_{1}]^{<\aleph_{1}}$. The model $\mathcal{N}$ is the permutation model determined by $M$, $\mathcal{G}$ and $\mathcal{F}$. We note that, if $x \in \mathcal{N}$, then there exists $E = \bigcup\{A_{i}: i \in \mathcal{I}\}$ for some $\mathcal{I} \in [\aleph_{1}]^{<\aleph_{1}}$ such that fix$_{\mathcal{G}}(E) \subseteq Sym_{\mathcal{G}}(x)$. Any such set $E \subseteq A$ is called a {\em support} of $x$.

\begin{claim}
\textbf{Form 233} holds in $\mathcal{N}$.
\end{claim}
\begin{proof}
Fix a field $\mathcal{K}'$ in $\mathcal{N}$. Let $\mathcal{K}$ be an algebraic closure of $\mathcal{K}'$ in $\mathcal{N}$ with support $E=\bigcup\{A_{i} : i \in K\}$ for some $K\in [\aleph_{1}]^{<\aleph_{1}}$. We show that $\mathcal{K}$ is well-orderable in $\mathcal{N}$.
Otherwise, there is a $x\in\mathcal{K}$ and a $\phi \in $ fix$_{\mathcal{G}}(E)$ with $\phi(x)\neq x$.
Under such assumptions, Tachtsis constructed a permutation $\psi \in$ fix$_{\mathcal{G}}(E)$ such
that $\psi(x) \neq x$ but $\psi^{2}$ is the identity mapping (cf. the proof of $\mathsf{LW}$ (every linearly ordered set can be well-ordered \cite[\textbf{Form 90}]{HR1998}) in $\mathcal{N}$ from \cite[claim 4.10]{Tac2019a}).
The permutation $\psi$ induces an automorphism of $\mathcal{K}$ and we can therefore apply Fact 5.1 to conclude that $\psi(i) = - i \neq i$ for some $i =\sqrt{-1}\in\mathcal{K}$. 

In order to obtain a contradiction, we prove that for every $\pi\in$ {fix}$ _{\mathcal{G}}(E)$, $\pi(i)=i$ for every $i=\sqrt{-1} \in\mathcal{K}$.
Fix an $i=\sqrt{-1} \in\mathcal{K}$. It is enough to show that $E$ is a support of $i$. We note that $i$ is a solution to the equation $x^{2} + 1 = 0$ all of whose coefficients are fixed by any $\eta\in$ fix$_{\mathcal{G}}(E)$. So if $\eta\in$ fix$_{\mathcal{G}}(E)$, then $\eta(i)$ is also a solution to $x^{2}+ 1 = 0$. Suppose $E$ is not a support of $i$. We follow the
subsequent three steps to complete the proof.

\textbf{Step 1:} We follow the ideas due to Tachtsis from \cite[proof of Lemma 1]{Tac2016a} to show that {\em if $x \in \mathcal{N}$ and $E_{1}, E_{2}$ are supports of $x$, then $E_{1} \cap E_{2}$ is a support of $x$, i.e., $\mathcal{N}$ satisfies Mostowski’s intersection lemma} by making the minor, necessary, modifications.

Let $x \in \mathcal{N}$ and $E_{1} = \bigcup\{A_{i} : i \in K_{1}\}, E_{2} = \bigcup\{A_{i} : i \in K_{2}\}$, where $K_{1},K_{2}\in [\aleph_{1}]^{<\aleph_{1}}$, be two supports of $x$. Let $\phi \in$ fix$_{\mathcal{G}}(E_{1}\cap E_{2})$, where
$E_{1} \cap E_{2} = \bigcup\{A_{i} : i \in K_{1} \cap K_{2}\}$. We will show that $\phi(x) = x$. 
In particular, we construct three permutations $\rho_{1}\in$ fix$_{\mathcal{G}}(E_{1})$, $\rho_{2}\in$ fix$_{\mathcal{G}}(E_{2})$, and $\phi'\in$ fix$_{\mathcal{G}}(E_{1}\cup E_{2})$ such that $\phi=\rho^{-1}_{1}\rho^{-1}_{2}\phi'\rho_{2}\rho_{1}$. Then $\phi(x)=\rho^{-1}_{1}\rho^{-1}_{2}\phi'\rho_{2}\rho_{1}(x)=x$, and we are done. 

To this end, we first let $\mathcal{A} = \{A_{i}: i < \aleph_{1}\}$ and, for a set $x \subseteq A$, we let $tr(x) = \{i \in \aleph_{1}:A_{i} \cap x \neq \emptyset\}$.
Put $W = \{a \in A : \phi(a) \neq a\}$ and 
$W^{*} = \bigcup\{A_{i}: i \in tr(W)\}$. 
Note that, from the definition of $\mathcal{G}$, $W$ is countable, and thus so is $W^{*}$. Let $\mathcal{U}, \mathcal{V}$ be two disjoint subsets of $\mathcal{A}$ such that each of the sets $U =\bigcup \mathcal{U}$ and $V =
\bigcup\mathcal{V}$ is disjoint from $E_{1} \cup E_{2} \cup W^{*}$, $tr(U)$ has the same order type as $K_{2} \backslash K_{1}$ and $tr(V)$ has the same order type as $(K_{1} \cup tr(W)) \backslash K_{2}$.
Consider two bijections $f : U \rightarrow E_{2}\backslash E_{1}$ and $g : V \rightarrow (E_{1} \cup W^{*})\backslash E_{2}$ so that, if $i \in tr(U)$ and $j \in tr(V)$, then 
$f[A_{i}] \in \{A_{k} : k \in K_{2}\backslash K_{1}\}$ and $g[A_{j}] \in \{A_{k} : k \in (K_{1} \cup tr(W)) \backslash K_{2}\}$.
We define
\begin{center}
$\rho_{1}=\prod_{u\in U}(u, f(u))$ and $\rho_{2}= \prod_{v\in V} (v, g(v))$,     
\end{center}

i.e., each one of $\rho_{1}$ and $\rho_{2}$ as a product of disjoint transpositions. Define $\phi'=\rho_{2}\rho_{1}\phi\rho^{-1}_{1}\rho^{-1}_{2}$. 
We can see that $\rho_{1}\in$ fix$_{\mathcal{G}}(E_{1})$ and $\rho_{2}\in$ fix$_{\mathcal{G}}(E_{2})$. In order to see that $\phi'\in$ fix$_{\mathcal{G}}(E_{1}\cup E_{2})$, let $e\in E_{1}\cup E_{2}$. We consider three cases $e\in E_{2}\backslash E_{1}$ or $e\in E_{1}\cap E_{2}$ or $e\in E_{1}\backslash E_{2}$. In each one of these cases, we can see that $\phi'(e)=e$. This completes the proof of step 1.

Let $E'$ be a support of $i$ and let $F = E' \backslash E$. Then $F \neq \emptyset$ (since $E$ is not a support of $i$) and $F \cap E = \emptyset$. Without loss of generality, we may assume that $E\subsetneq E'$. 

\textbf{Step 2:} Using step 1, we prove that {\em if $\phi, \phi' \in$ fix$_{\mathcal{G}}(E)$ and $\phi(F) \cap \phi'(F) = \emptyset$, then $\phi(i) \neq \phi'(i)$}. 
For the sake of contradiction assume that $\phi, \phi' \in$ fix$_{\mathcal{G}}(E)$ such that $\phi(F) \cap \phi'(F) = \emptyset$ and $\phi(i)=\phi'(i)$. Then $(\phi')^{-1}\phi(i)=i$. Since $E'$ supports $i$, we have $(\phi')^{-1}\phi(E')$ supports $(\phi')^{-1}\phi(i)=i$. By step 1, $(\phi')^{-1}\phi(E')\cap E'$ supports $i$. However, $(\phi')^{-1}\phi(E')\cap E'=E$ since $\phi(F) \cap \phi'(F) = \emptyset$ and $\phi, \phi' \in$ fix$_{\mathcal{G}}(E)$, which contradicts the assumption that $E$ is not a support of $i$. 

\textbf{Step 3:} Using step 2 and the features of $\mathcal{N}$, we can observe that {\em there is a set $S = \{\phi_{k}(i) : k \in \omega\}$ in $\mathcal{N}$ such that for every $k \in \omega$, $\phi_{k}\in$ fix$_{\mathcal{G}}(E)$, and for all $k,l \in \omega$, if $k \neq l$ then $\phi_{k}(i) \neq \phi_{l}(i)$.}

In particular, we consider a denumerable collection $\{B_{i}:i\in \omega\}$ of subsets of $A\backslash (E\cup F)$ where $B_{i}=\bigcup\{A_{j}:j\in K_{i}\}$ for some $K_{i}\in [\aleph_{1}]^{<\aleph_{1}}$ such that $\vert K_{i}\vert= \vert tr(F)\vert$ for every $i\in\omega$, and
$B_{i}\cap B_{j}=\emptyset$ if $i\neq j$. 
For every $k \in \omega$, let $H_{k} : F \rightarrow B_{k}$ 
be a bijection such that, for every $i \in tr(F)$, $H_{k}[A_{i}] = A_{j}$ for some $j \in K_{k}$, and also let $\phi_{k} = \prod_{a\in F}(a, H_{k}(a))$.
Then $\phi_{k}(F)=B_{k}$.
Fix any $k,l\in\omega$ such that $k\neq l$. 
By step 2, $\phi_{k}(i) \neq \phi_{l}(i)$ since $B_{k}\cap B_{l}=\emptyset$. Since, $E'$ is a support of $i$, $\phi_{k}(E')$ is a support of $\phi_{k}(i)$ for every $k\in\omega$. Then $\bigcup _{k\in\omega}\phi_{k}(E')$ is a support of $S = \{\phi_{k}(i): k \in \omega\}$ since the set of supports is closed under countable unions. Consequently, the set $S$ is in $\mathcal{N}$ and $S$ is infinite (in fact, countably infinite).

So, the equation $x^{2} + 1 = 0$ has infinitely many solutions in $\mathcal{K}$, which is a contradiction. 
Thus, $E$ is a support of $i$.

Since $\mathcal{K}$ is well-orderable in $\mathcal{N}$, we can use transfinite induction without using any form of choice to finish the proof. 
\end{proof}

\begin{claim}
$W_{\aleph_{1}}$ fails in $\mathcal{N}$.
\end{claim}
\begin{proof}
First, we show that (Sym($A$)) = $\aleph_{0}$Sym($A$) in $\mathcal{N}$. For the sake of contradiction, assume $f\in ($Sym$(A))\backslash \aleph_{0}$Sym($A)$.
Let $E = \bigcup\{A_{i} : i \in \mathcal{I}\}, \mathcal{I} \in [\aleph_{1}]^{<\aleph_{1}}$ be a support of $f$. Then there exists $i \in \aleph_{1}\backslash \mathcal{I}$ such that $a \in A_{i}$, $b \in A\backslash (E \cup \{a\})$ and $b = f(a)$.
Let $b \in A_{i}$. Consider $\phi\restriction A_{i}$ such that  $\phi\restriction A_{i}$ moves every atom in $A_{i}$ except $b$ and $\phi\restriction A\backslash A_{i} = 1_{A\backslash A_{i}}$. Clearly,  $\phi\in\mathcal{G}$. Also, $\phi(b)=b$, $\phi\in$ fix$_{\mathcal{G}}(E)$, and hence $\phi(f) = f$. Thus
\begin{center}
$(a,b)\in f\implies (\phi(a), \phi(b))\in \phi(f)\implies (\phi(a), b)\in \phi(f)=f$.   
\end{center}

So $f$ is not injective; a contradiction. 
If $b \in A \backslash (E \cup A_{i})$, then 
consider $\phi\restriction A_{i}$ such that  $\phi\restriction A_{i}$ moves every atom in $A_{i}$ and $\phi\restriction A\backslash A_{i} = 1_{A\backslash A_{i}}$. Again $\phi\in\mathcal{G}$, and we easily obtain a contradiction.

We note that $|A|\not\leq\aleph_{1}$
in $\mathcal{N}$.
In order to show that $W_{\aleph_{1}}$ fails, we prove that there is no injection $f:\aleph_{1}\rightarrow A$. Assume there exists such an $f$ and suppose $\{y_{n}\}_{n\in \aleph_{1}}$ is an enumeration of the elements of $Y=f(\aleph_{1})$. 
We can use transfinite recursion, without using any form of choice, to construct a bijection $h: Y\rightarrow Y$ such that $h(x)\neq x$ for any $x\in Y$. Define $g:A \rightarrow A$ as follows: 
$g(x)=h(x)$ if $x\in Y$, and $g(x)=x$ if $x\in A\backslash Y$.
Clearly $g \in$ Sym($A$)$\backslash$ $\aleph_{0}$Sym($A$), and hence Sym($A$) $\neq $ $\aleph_{0}$Sym($A$); a contradiction.
\end{proof}

(2). Consider the permutation model of (1) (say $\mathcal{N}$) by replacing $\aleph_{0}$ and $\aleph_{1}$ with $\aleph_{1}$ and $\aleph_{2}$ respectively. Following the arguments of \cite[Theorem 4.7]{Tac2019a}, we can see that $\mathsf{AC^{LO}}$ holds in $\mathcal{N}$, but $\mathsf{LT}$ fails. By the arguments of the previous proof, $\mathsf{W_{\aleph_{2}}}$ fails and \textbf{Form 233} holds in $\mathcal{N}$.
Moreover, $\mathsf{DC_{\aleph_{1}}}$ holds since $\mathcal{I}$ is closed under $<\aleph_{2}$ unions (cf. \cite[the arguments in the proof of Theorem 8.3 (i)]{Jec1973}). Consequently, $\mathsf{EDM}$ holds by Theorem 4.1(1).
\end{proof}

\section{Concluding remarks and questions}
\subsection{Remarks}
(1). Recently, Banerjee \cite{Ban2023} and Karagila \cite{Kar2019} proved that if $V$ is a model of $\mathsf{ZFC}$, then $\mathsf{DC}_{\aleph_{1}}$ can be preserved in the symmetric extension $\mathcal{N}$ of $V$ (symmetric submodel of a forcing extension where $\mathsf{AC}$ can consistently fail) if the forcing notion $\mathbb{P}$ is either $\aleph_{2}$-distributive or $\aleph_{2}$-c.c., $\mathcal{G}$ is any group of automorphisms of $\mathbb{P}$, and the normal filter $\mathcal{F}$ of subgroups over $\mathcal{G}$ is $\aleph_{2}$-complete. By Theorem 4.1, $\mathsf{EDM}$ holds in $\mathcal{N}$.

(2). By Theorems 4.1(1,3), and the facts that $\mathsf{DC_{\aleph_{1}}}$ does not imply $\mathsf{LT}$ and $\mathsf{LT}$ does not imply $\mathsf{RT}$ in $\mathsf{ZF}$ (cf. \cite[Theorems 4.3, 4.13]{Tac2019a}), $\mathsf{LT}$ and $\mathsf{EDM}$ are mutually independent in $\mathsf{ZF}$.

(3). Fix $X\in\{\mathsf{BPI},\mathsf{KW},\mathsf{AC_{WO}}$, 2-coloring theorem, ``There are no amorphous sets"$\}$.
Blass \cite{Bla1977} proved that
$\mathsf{RT}$ is false in the basic Cohen model (Model $\mathcal{M}_{1}$ in \cite{HR1998}) where $X$ holds. Following Theorem 4.1(3), $X$ does not imply $\mathsf{EDM}$ in $\mathsf{ZF}$. On the other hand, $X$ fails in $\mathcal{N}_{1}$.
Thus $\mathsf{EDM}$ and $X$ are mutually independent in $\mathsf{ZFA}$ by Theorem 4.2(1).

(4). Fix $X\in \{\mathsf{A},\mathsf{WOAM}, \mathsf{CS}\}$. We can see that $X$ and $\mathsf{EDM}$ are mutually independent in $\mathsf{ZFA}$. Following Theorem 4.4 and the fact that $X$ holds in $\mathcal{N}_{HT}^{1}(2),$\footnote{Tachtsis proved $\mathsf{WOAM}$ holds in $\mathcal{N}_{HT}^{1}(2)$, see \cite[Lemma 2]{Tac2016a}.} $X$ does not imply $\mathsf{EDM}$ in $\mathsf{ZFA}$. The other direction follows from Theorem 4.2(3).

(5). Consider the permutation model $\mathcal{M}$ from Corollary 3.6 where $\mathsf{CAC^{\aleph_{0}}}$ fails (and thus $\mathsf{EDM}$ fails by Theorem 4.1) and $\mathsf{DT}$ holds. Secondly, we consider the permutation model $\mathcal{V}$ from \cite[Theorem 9(4)]{Tac2022} where $\mathsf{DT}$ fails and $\mathsf{DC_{\aleph_{1}}}$ holds, and hence  $\mathsf{EDM}$ and $\mathsf{CAC^{\aleph_{0}}}$ hold as well. Thus if $X\in\{\mathsf{EDM}, \mathsf{CAC^{\aleph_{0}}}\}$, then $X$ and $\mathsf{DT}$ are mutually independent in $\mathsf{ZFA}$. 

(6).  Following the arguments of \cite[Theorem 4(11)]{Tac2022} due to Tachtsis (where he proved that $\mathsf{CAC_{1}^{\aleph_{0}}}$ implies $\mathsf{CAC}$ in $\mathsf{ZF}$) we can see that $\mathsf{CAC^{\aleph_{0}}}$ implies $\mathsf{CAC}$ in $\mathsf{ZF}$. By Theorem 4.1(4), $\mathsf{CAC}$ does not imply $\mathsf{CAC^{\aleph_{0}}}$ in $\mathsf{ZF}$ since $\mathsf{DF=F}$ implies $\mathsf{CAC}$ in $\mathsf{ZF}$.

(7). The referee remarked that $\mathsf{CACT_{1}^{\aleph_{0}}}+\mathsf{CACT^{\aleph_{0}}}\not\rightarrow\mathsf{CAC}$ in $\mathsf{ZFA}$. 
In particular, by \cite[Theorem 6]{Tac2022} due to Tachtsis, $\mathsf{AC_{WO}}$ (i.e., the axiom of choice for families of non-empty, well-orderable sets) does not imply $\mathsf{CAC}$ in $\mathsf{ZFA}$, and thus neither does $\mathsf{AC_{fin}^{\aleph_{0}}}$ imply $\mathsf{CAC}$ in $\mathsf{ZFA}$. This, together with Proposition 3.3(5) yields $\mathsf{CACT_{1}^{\aleph_{0}}}+\mathsf{CACT^{\aleph_{0}}}$ does not imply $\mathsf{CAC}$ in $\mathsf{ZFA}$.

(8). In the second Fraenkel model $\mathcal{N}_{2}$, $\mathsf{AC_{fin}^{\aleph_{0}}}$ fails but $\mathsf{MC}$ holds. By Theorem 4.1, if $X\in \{\mathsf{EDM},\mathsf{CAC_{1}^{\aleph_{0}}}, \mathsf{CAC^{\aleph_{0}}}\}$ then $\mathsf{MC}$ does not imply $X$ in $\mathsf{ZFA}$. Since $\mathsf{MC}$ fails in $\mathcal{N}_{1}$, $\mathsf{EDM}$ and $\mathsf{MC}$ are mutually independent in $\mathsf{ZFA}$ by Theorem 4.2. 

\subsection{Questions}
\begin{question}
Does $\mathsf{CAC^{\aleph_{0}}}$  imply $\mathsf{CAC_{1}^{\aleph_{0}}}$ in $\mathsf{ZF}$ or in $\mathsf{ZFA}$? 
\end{question}

\begin{question}
Does $\mathsf{BPI}+\mathsf{DC}$ imply $\mathsf{EDM}$ in $\mathsf{ZF}$ or in $\mathsf{ZFA}$?
\end{question}

\begin{question}Does $\mathsf{WOAM}$ imply $\mathsf{CAC^{\aleph_{0}}}$ in $\mathsf{ZF}$ or in $\mathsf{ZFA}$?
\end{question}

\begin{question}
Does $\mathsf{AC^{LO}}$ imply $\mathsf{EDM}$ in $\mathsf{ZFA}$?
\end{question}

\begin{question}
Does $\mathsf{EDM}$ hold in Brunner/Pincus’s
Model ($\mathcal{N}_{26}$ in \cite{HR1998})?

\begin{question}Does either of $\mathsf{CACT^{\aleph_{0}}}$ and $\mathsf{CACT_{1}^{\aleph_{0}}}$ implies $\mathsf{AC_{fin}^{\aleph_{0}}}$ in $\mathsf{ZF}$ or in $\mathsf{ZFA}$? 

\end{question}
\begin{question}
Does $\mathsf{AC_{fin}^{\aleph_{0}}}$ hold in Pincus's Model $X$ ($\mathcal{N}_{34}$ in \cite{HR1998})?
\end{question}

\end{question}
\section{Acknowledgements}
The authors would like to thank Prof. Lajos Soukup and Prof. Eleftherios Tachtsis for several discussions concerning Kurepa's result ($\mathsf{CAC_{1}^{\aleph_{0}}}$) and its variants, which have been the chief motivation for us in continuing
the research on this intriguing topic.
The authors are grateful to the anonymous referee for reading the manuscript in detail and for providing several comments and suggestions which improved the quality
and the exposition of the paper. We are especially thankful to the referee for Remark 4.5, Remark 6.1(7), and for outlining the three main steps in the proof of claim 5.3, which gave us the motivation to fill in the details.

\end{document}